\documentclass[11pt,a4paper]{amsart}
\usepackage[active]{srcltx}
\usepackage{graphicx}
\usepackage{amssymb}
\usepackage{stmaryrd}
\usepackage{enumerate}
\usepackage[T1]{fontenc}
\usepackage[latin1]{inputenc}
\usepackage[curve,color,line]{xypic}
\usepackage[x11names]{xcolor}
\usepackage{soul}
\usepackage{hyperref}
\usepackage{tikz}

\setstcolor{red}

\renewcommand{\geq}{\geqslant}
\renewcommand{\leq}{\leqslant}
\renewcommand{\ge}{\geqslant}
\renewcommand{\le}{\leqslant}
\let\op=\llbracket
\let\cl=\rrbracket
\def\pv#1{\ensuremath{\mathsf{#1}}}
\def\Om#1#2{\ensuremath{\overline{\Omega}_{#1}{\pv{#2}}}}
\def\oms#1#2{\ensuremath{\Omega^\sigma_{#1}{\pv{#2}}}}
\def\omo#1#2{\ensuremath{\Omega^\omega_{#1}{\pv{#2}}}}
\let\cal=\mathcal
\def\Cl#1{\ensuremath{\cal#1}}
\def\Pol#1{\ensuremath{\mathop{\mathrm{Pol}}\pv{#1}}}
\def\PolL#1{\ensuremath{\mathop{\mathrm{Pol}}\Cl{#1}}}

\def\PolB#1{\ensuremath{\mathop{\mathrm{Pol}\,\mathrm{B}}\pv{#1}}}
\newcommand\malcev{\mathbin{\bigcirc\kern-8.5pt%
\raise1pt\hbox{\footnotesize$m$}\kern1pt}}

\newtheorem{Thm}{Theorem}[section]
\newtheorem{Prop}[Thm]{Proposition}
\newtheorem{Lemma}[Thm]{Lemma}
\newtheorem{Cor}[Thm]{Corollary}
\newenvironment{claim}[1]{\par\vskip5pt plus1pt minus1pt
  \trivlist\item{\textbf{Claim.}}\space#1}{\endtrivlist}

\begin{document}
\title[The omega-inequality problem]{The omega-inequality
  problem for concatenation hierarchies of star-free languages}

\author{J. Almeida}%
\address{CMUP, Dep.\ Matem\'atica, Faculdade de Ci\^encias, Universidade do
  Porto, Rua do Campo Alegre 687, 4169-007 Porto, Portugal}
\email{jalmeida@fc.up.pt}

\author{O. Kl\'ima}%
\author{M. Kunc}%
\address{Dept.\ of Mathematics and Statistics, Masaryk University,
  Kotl\'a\v rsk\'a 2, 611 37 Brno, Czech Republic}%
\email{klima@math.muni.cz, kunc@math.muni.cz}

\begin{abstract}
  The problem considered in this paper is whether an inequality of
  $\omega$-terms is valid in a given level of a concatenation
  hierarchy of star-free languages. The main result shows that this
  problem is decidable for all (integer and half) levels of the
  Straubing-Th\'erien hierarchy.
\end{abstract}


\makeatletter%
\@namedef{subjclassname@2010}{%
  \textup{2010} Mathematics Subject Classification}%
\makeatother


\maketitle

\section{Introduction}
\label{sec:intro}

With the advent of computers in the 1950's, there was a surge of
interest in formal languages, among which regular languages play a key
role. After several examples of algebraization of questions on classes
(\emph{varieties}) of regular languages had been already discovered in
the 1960's and early 1970's, Eilenberg described a general framework
for the algebraic formulation of such questions, namely in terms of
what is commonly known as Eilenberg's correspondence
\cite{Eilenberg:1976}. The algebraic structures considered by
Eilenberg are finite semigroups and monoids and the classes of such
structures corresponding to varieties of languages are called
\emph{pseudovarieties}. In general, the hope is that the
algebraization will convert the membership problem for a given variety
of languages in a more manageable membership problem for the
corresponding pseudovariety, and indeed, the pioneering examples of
instances of the correspondence before its general formulation were
found with this purpose and effective application.

Eilenberg's correspondence has been extended in several directions, in
particular to capture more general classes of regular languages. One
such generalization was developed by Pin \cite{Pin:1995d} with the aim
of refining and better understanding a hierarchy of star-free
languages introduced by Brzozowski and
Cohen~\cite{Cohen&Brzozowski:1971} and its variant considered by
Straubing \cite{Straubing:1981a,Straubing:1985} and Th\'erien
\cite{Therien:1981a}. The algebraic structures emerging in this
context are finite semigroups and monoids with a compatible partial
order, the relevant classes of such structures being also called
pseudovarieties. While pseudovarieties of semigroups and monoids are
defined by so-called \emph{pseudoidentities}, which are formal
equalities of members of suitable free profinite structures
\cite{Reiterman:1982,Almeida:1994a}, in the ordered counterpart it
suffices to replace equality by formal inequality, leading to
\emph{pseudoinequalities}, to obtain a similar
result~\cite{Molchanov:1994,Pin&Weil:1996b}.

Additional motivation for investigating the Straubing-Th\'erien
hierarchy comes from logic, specifically from finite model theory.
Indeed, as has been shown by Thomas~\cite{Thomas:1982}, the
Straubing-Th\'erien hierarchy may be viewed as a skeleton of the
hierarchy of languages defined by sentences in the first-order
language with a binary predicate (for ordering positions of letters)
and unary predicates for the letters of the alphabet (to denote the
presence of a letter in a given position), where the complexity of
sentences in prenex normal form is measured in terms of the number of
quantifier alternations.

Although significant progress has been recently achieved on the
membership problem for the levels of the Straubing-Th\'erien hierarchy
\cite{Place&Zeitoun:2014b,Place&Zeitoun:2014a,
  Almeida&Bartonova&Klima&Kunc:2015}, it remains an open problem
whether it is decidable at all levels. A key tool that has been used
in such works is the following \emph{lifting} problem: to determine
when a pair of elements of a finite monoid can be realized as values
in the monoid of the sides of a pseudoinequality which is valid in a
given level of the corresponding hierarchy of pseudovarieties. The
decidability of this condition, which Place and Zeitoun
\cite{Place&Zeitoun:2014a} call the \emph{separation property},
entails the decidability of the membership problem for the
pseudovariety and it may be viewed as a form of
\emph{hyperdecidability}, namely for the inequality $x\le y$, in the
sense of~\cite{Almeida:1999b}. Note that in general there are
uncountably many pseudoinequalities to be considered as potential
liftings of a given pair of elements of a finite monoid.

A method introduced by Steinberg and the first author
\cite{Almeida&Steinberg:2000a} to approach such decidability questions
consists in solving two separate problems involving an \emph{implicit
  signature} $\sigma$ with suitable computational properties:
\begin{enumerate}
\item\label{item:prob-1} to solve the $\sigma$-inequality problem for
  the pseudovariety in question;
\item\label{item:prob-2} to show that, if an instance of the lifting
  problem has a solution, then it admits one with a pseudoinequality
  which is actually a $\sigma$-inequality.
\end{enumerate}
Although problems such as Problem~\eqref{item:prob-2} tend to be very
hard, they have the advantage of being abstract problems which are not
of an algorithmic nature. In the aperiodic case, which is a suitable
setting for the problems concerning the Straubing-Th\'erien hierarchy,
the most frequently considered implicit signature consists of
multiplication, 1 and $\omega$-power, and is also denoted~$\omega$.

The main result of this paper is a solution of
Problem~\eqref{item:prob-1} for the signature~$\omega$ for every level
of the Straubing-Th\'erien hierarchy. There are two key ingredients in
the proof of this result. The first is the fact that, for a
polynomially closed pseudovariety of ordered monoids \pv V satisfying
a pseudoinequality $u\le v$, any factorization of~$u$ induces a
factorization of~$v$ of the same length such that the inequality
remains valid in~\pv V for factors in the same position (Section 3).
The second is a syntactic analysis of $\omega$-terms that leads to a
description of its factors (Section 5) and to a repetition or
periodicity result (Section 6). These may be viewed as finiteness
properties of $\omega$-words over the pseudovariety of all finite
aperiodic monoids. For further such properties and generalizations,
see \cite{Almeida&Costa&Zeitoun:2009a,Gool&Steinberg:2016}. The
combination of these results allows us to prove a sort of completeness
theorem for $\omega$-inequalities valid in the Boolean-polynomial
closure of any polynomially closed pseudovariety of aperiodic ordered
monoids: there is a complete deductive calculus to obtain all such
inequalities from a well-determined basis (Section 7). Once such a
calculus has been established, one may effectively enumerate the
consequences of a recursively enumerable basis. Since the
$\omega$-inequalities which are not consequences of the basis fail in
concrete finite models of the basis, decidability of the
$\omega$-inequality problem follows provided both the basis and the
pseudovariety of ordered monoids are recursively enumerable. Putting
it all together, this yields an inductive argument that shows that all
levels of the Straubing-Th\'erien hierarchy have decidable
$\omega$-inequality problem (Section~8).

\section{Preliminaries}
\label{sec:prelims}

The reader is referred to the standard literature on finite semigroups
for general background \cite{Almeida:1994a,Rhodes&Steinberg:2009qt}.
Nevertheless, we recall here some basic notions for the sake of
completeness.

\subsection{Pseudovarieties}
\label{sec:pseudovarieties}

A \emph{pseudovariety} of monoids is a class of finite monoids that is
closed under taking homomorphic images, submonoids and finite
direct products. For finite ordered monoids, that is, monoids equipped
with a compatible partial order, pseudovarieties are defined in the same way.
A~pseudovariety \pv W of monoids may be identified with the pseudovariety
of ordered monoids consisting of the members of~\pv W equipped with all
possible compatible partial orders. Then, for a pseudovariety \pv V
of ordered monoids, the class of monoids
obtained by forgetting the order of all elements of~\pv V and taking
homomorphic images is a
pseudovariety of monoids, whence it is the
pseudovariety of monoids generated by~\pv V. This is also
the join of~\pv V with its \emph{dual}, consisting of the members
of~\pv V with the orders reversed.

The \emph{trivial} pseudovariety consists of all singleton monoids.
Two further examples of pseudovarieties of monoids are \pv M, consisting
of all finite monoids, and \pv A, consisting of all finite aperiodic
monoids, that is, finite monoids all of whose subgroups are trivial.

By a \emph{pro-\pv V monoid} we mean a monoid with a compact topology
for which multiplication is continuous such that distinct points may
be separated by continuous homomorphisms into members of~\pv V, these
being regarded as discrete spaces. For a pseudovariety \pv V of
[ordered] monoids and a finite set~$A$, the free pro-\pv V monoid
on~$A$ is denoted \Om AV. It may be constructed as the inverse limit
of the natural projective family of $A$-generated monoids from~\pv V.
For our purposes, the relevant property is the universal property that
justifies its name: every function from $A$ to a pro-\pv V monoid $M$
extends uniquely to a continuous homomorphism on~\Om AV, as depicted
in the following diagram:
$$\xymatrix{
  A %
  \ar[r]^(.45)\iota %
  \ar[rd] %
  &
  \Om AV %
  \ar@{-->}[d] %
  \\
  &
  \,M\,.
}$$
This property
of~\Om AV entails that each homomorphism $\varphi \colon A^*\to M$
into a pro-\pv V monoid $M$ extends
uniquely to a continuous homomorphism $\hat\varphi \colon \Om AV\to
M$, leading to the following commutative diagram:
$$\xymatrix@H=11pt{
  A %
  \ar[r]^(.45)\iota %
  \ar@{^(->}[d] %
  &
  \Om AV %
  \ar[d]^{\hat\varphi} %
  \\
  A^* %
  \ar[ru]^{\hat\iota} %
  \ar[r]^\varphi
  &
  \,M\,,
}$$
where $\hat\iota:A^*\to\Om AV$ is the unique extension of $\iota$ to a
homomorphism.
In case $\alpha \colon A^*\to B^*$ is a homomorphism, we may also
view it as a homomorphism $A^*\to\Om BV$, and consider its unique
extension to a continuous homomorphism $\hat\alpha \colon \Om AV\to\Om
BV$.

Often, the natural monoid homomorphism $\hat\iota:A^*\to\Om AV$ is an embedding,
in which case we identify $A^*$ with its image. This is the case, for
instance, for all pseudovarieties of ordered monoids containing the
pseudovariety \pv{MN} consisting of all finite nilpotent semigroups
with an identity element adjoined, a property which holds for most
nontrivial pseudovarieties of ordered monoids considered in this paper
and which will therefore be used freely. Moreover, for such
pseudovarieties, the topology on $A^*$ induced from~\Om AV is
discrete.

The elements of~\Om AV are sometimes called \emph{pseudowords over~\pv
  V}. A \emph{pseudoinequality over~\pv V} is a formal inequality of
two pseudowords over~\pv V; in what follows, they are simply called
inequalities. The inequality $u\le v$ is \emph{trivial} if $u$ and $v$
coincide. For $u,v\in\Om AV$ and $S\in\pv V$, the inequality $u\le v$
is said to \emph{hold} in or to be \emph{satisfied} by~$S$ if
$\varphi(u)\le\varphi(v)$ for every continuous homomorphism
$\varphi:\Om AV\to S$, and to \emph{hold} in a subclass \Cl C of~\pv V
if it holds in every member of~\Cl C; the \emph{pseudoidentity} $u=v$
is \emph{satisfied} by~\Cl C if both inequalities $u\le v$ and $v\le
u$ are satisfied by the members of~\Cl C. It is
well known that, if an inequality $u\le v$ holds in~\pv{MN}, that is,
the pseudoidentity $u=v$ holds in~\pv{MN}, and either $u$ or $v$ are
words, then it is trivial.

Note that, for a pseudovariety of ordered monoids \pv V, \Om AV is an
ordered monoid for the relation $\le$ such that
$u\le v$ if and only if the
inequality $u\le v$ over~\pv V holds in~\pv V.

For a pseudovariety \pv V of ordered monoids, a language $L\subseteq
A^*$ is said to be \emph{\pv V-recognizable} if there exist a
homomorphism $\varphi \colon A^*\to M$ into a monoid $M$ from~\pv V and
an up-closed (that is, an order filter) subset $F$ of~$M$ such that
$L=\varphi^{-1}(F)$. Equivalently, the \emph{syntactic ordered monoid}
$\mathrm{Synt}(L)$ of~$L$ belongs to~\pv V, where $\mathrm{Synt}(L)$
is the quotient of~$A^*$ by the congruence $\le_L\cap\ge_L$, ordered
by the partial order induced by the quasi-order~$\le_L$, where
$u\le_L v$ is defined by the following condition: for every $x,y\in
A^*$, $xuy\in L$ implies $xvy\in L$. The natural homomorphism
$\varphi \colon A^*\to\mathrm{Synt}(L)$ as well as its extension
$\hat\varphi$ are both called \emph{syntactic homomorphisms}. It
should be noted that, in several papers by Pin and coauthors, such
as~\cite{Pin&Weil:1994c}, the syntactic order $\le_L$ is defined to be
the dual of the order considered here; see
\cite{Almeida&Cano&Klima&Pin:2015} for an explanation as to why our
choice should be preferred.

Yet another equivalent formulation of \pv V-recognizability
of~$L\subseteq A^*$ is that its topological closure $\overline{L}$
in~\Om AV is open. If $\varphi \colon A^*\to M$ is a homomorphism into
$M\in\pv V$ that recognizes the language $L$, then $\hat\varphi \colon
\Om AV\to M$ recognizes~$\overline{L}$, in the sense that
$\overline{L}=\hat\varphi^{-1}(\hat\varphi(\overline{L}))$.
Note that, if $L,K\subseteq A^*$ are \pv
V-recognizable languages, then the open set
$\overline{L}\cap\overline{K}$ coincides with $\overline{L\cap K}$.

For pseudovariety \pv V of ordered monoids, the \pv V-recognizable
languages constitute what is known as a \emph{positive variety of
  languages}. More precisely, a positive variety of languages is
defined as a correspondence \Cl V associating with each finite
alphabet $A$ a set $\Cl V(A)$ of regular languages over~$A$ that is
closed under finite union, finite intersection, and the derivative
operations $L\mapsto a^{-1}L=\{w:aw\in L\}$, $L\mapsto
La^{-1}=\{w:wa\in L\}$, and such that $L\in\Cl V(B)$ implies
$\varphi^{-1}(L)\in\Cl V(A)$ whenever $\varphi:A^*\to B^*$ is a
homomorphism. The above defined correspondence $\pv V\mapsto\Cl V$ is
a bijection \cite{Pin:1995d} and it is the natural analog for ordered
monoids of Eilenberg's Correspondence Theorem \cite{Eilenberg:1976},
relating pseudovarieties of monoids and varieties of languages.

\subsection{Polynomial closure}
\label{sec:poly-closure}

Given a positive variety of languages \Cl V, its \emph{polynomial
closure} \PolL V is defined by letting $\PolL V(A)$ consist of all
unions of finitely many languages of the form
$$L_0a_1L_1\cdots a_nL_n,$$
where the $L_i$ belong to $\Cl V(A)$ and the $a_i\in A$ are letters. 
Note that the equality $\PolL{}{\PolL V}=\PolL V$
follows directly from the definition. 
We say that \Cl V is  
\emph{polynomially closed} if $\Cl V= \PolL V$.
It is well known that \PolL V is a positive variety of languages
whenever \Cl V is a variety of languages. Therefore, for 
the pseudovariety \pv V of monoids corresponding to a variety of languages
$\Cl V$, we denote by $\Pol V$ the pseudovariety of ordered monoids corresponding
to the positive variety of languages \PolL V.
A pseudovariety of ordered monoids corresponding to
a polynomially closed positive variety of languages
is also called \emph{polynomially closed}.
Let us remark that $\Pol V$ is not defined
for a pseudovariety of ordered monoids $\pv V$ in general. 

By a result of Pin and Weil
\cite[Theorem~5.9]{Pin&Weil:1994c}, if \pv V is a pseudovariety of
monoids, then \Pol V may be described as the Mal'cev product $\op
x^\omega\le x^\omega yx^\omega\cl\malcev\pv V$.
Combining with a general basis theorem for Mal'cev products, also due
to Pin and Weil \cite{Pin&Weil:1996a}, we obtain a basis of
inequalities for \Pol V (see \cite[Proposition~7.4]{Pin:1997}).

\begin{Prop}
  \label{p:basis-Pol}
  Let \pv V be a pseudovariety of monoids. Then, \Pol V is defined by
  all inequalities of the form $u^{\omega}\le u^\omega vu^\omega$,
  with $u,v \in \Om AM$,
  such that the pseudoidentities $u=v=v^2$ hold in~\pv V.
\end{Prop}

Also of interest is to consider the \emph{Boolean-polynomial closure}
of a pseudovariety \pv V of ordered monoids, denoted \PolB V. This
means taking first the pseudovariety of monoids generated by~\pv V,
which recognizes precisely the languages over a finite alphabet $A$
that are Boolean combinations of languages recognized by~\pv V,
and then taking the polynomial closure of that pseudovariety of
monoids. The following basis of inequalities for $\PolB V$ can be
found in~\cite{Almeida&Bartonova&Klima&Kunc:2015}.

\begin{Prop}
  \label{p:basis-PolB}
  Let \pv W be a pseudovariety of monoids and let $\pv V=\Pol W$.
  Then, \PolB V is defined by all inequalities of the form
  $u^{\omega+1}\le u^\omega vu^\omega$, with $u,v \in \Om AM$, such
  that the inequality $v\le u$ holds in~\pv V.
\end{Prop}

The \emph{concatenation hierarchy} based on a pseudovariety of monoids
\pv V is the sequence starting at~\pv V that alternates polynomial
closure with forgetting order plus taking homomorphic images; the
pseudovarieties of monoids obtained by forgetting order plus taking
homomorphic images are called \emph{levels} of the hierarchy, while
the pseudovarieties of ordered monoids obtained by the polynomial
closure are called \emph{half levels}. The concatenation hierarchy
based on the trivial pseudovariety is known as the
\emph{Straubing-Th\'erien hierarchy}.

\subsection{Validity of inequalities}
\label{sec:ineq}

The following lemma gives two alternative characterizations of when an
inequality holds in a syntactic ordered monoid.

\begin{Lemma}
  \label{l:inequalities-for-syntactic}
  Let $A$ and $B$ be finite alphabets, $L\subseteq B^*$ be a regular
  language, and $u,v\in\Om AM$ be pseudowords. The following
  conditions are equivalent:
  \begin{enumerate}
  \item\label{item:inequalities-for-syntactic-1} the inequality $u\le
    v$ holds in the syntactic ordered monoid $\mathrm{Synt}(L)$;
  \item\label{item:inequalities-for-syntactic-2} for every homomorphism
    $\alpha \colon A^* \to B^*$ and all words $x,y\in B^*$,
    $x\hat{\alpha}(u)y\in\overline{L}$ implies
    $x\hat{\alpha}(v)y\in\overline{L}$;
  \item\label{item:inequalities-for-syntactic-3} for every homomorphism
    $\alpha \colon A^* \to B^*$ and all pseudowords $x,y\in\Om BM$,
    $x\hat{\alpha}(u)y\in\overline{L}$ implies
    $x\hat{\alpha}(v)y\in\overline{L}$.
  \end{enumerate}
\end{Lemma}

\begin{proof}
  Let $\varphi\colon\Om BM\to\mathrm{Synt}(L)$ be the syntactic homomorphism.
  
  $(\ref{item:inequalities-for-syntactic-1})
  \Rightarrow(\ref{item:inequalities-for-syntactic-3})$ %
  Let $\alpha \colon A^* \to B^*$ be a homomorphism and suppose that
  $x,y\in\Om BM$ are such that $x\hat{\alpha}(u)y\in\overline{L}$.
  Applying $\varphi$ and taking into account that
  $\varphi(\overline{L})=\varphi(L)$, we obtain the relation
  $\varphi(x\hat{\alpha}(u)y)\in\varphi(L)$. Recall that $\varphi(L)$ is an
  order filter in~$\mathrm{Synt}(L)$. Since the order
  in~$\mathrm{Synt}(L)$ is stable and
  $\varphi\hat{\alpha}(u)\le\varphi\hat{\alpha}(v)$
  by~(\ref{item:inequalities-for-syntactic-1}), it follows that
  $\varphi(x\hat{\alpha}(v)y)\in\varphi(L)$. Since $\varphi$ recognizes
  $\overline{L}$, we deduce that $x\hat{\alpha}(v)y\in\overline{L}$, as
  required.

  The implication $(\ref{item:inequalities-for-syntactic-3})
  \Rightarrow(\ref{item:inequalities-for-syntactic-2})$ %
  is trivial.

  $(\ref{item:inequalities-for-syntactic-2})
  \Rightarrow(\ref{item:inequalities-for-syntactic-1})$ %
  Let $\psi\colon\Om AM\to\mathrm{Synt}(L)$ be an arbitrary continuous
  homomorphism. We need to show that $\psi(u)\le\psi(v)$.
  Choose a homomorphism $\alpha \colon A^* \to B^*$ such that
  $\psi = \varphi \circ \hat{\alpha}$.
  By definition of the syntactic order and since
  $\varphi|_{B^*}$ recognizes $L$, the preceding inequality is
  equivalent to the property that, for all $p,q\in\mathrm{Synt}(L)$,
  $p\psi(u)q\in\varphi(L)$ implies $p\psi(v)q\in\varphi(L)$. To
  establish this property, note first that, since $\varphi|_{B^*}$ is
  onto, given $p,q\in\mathrm{Synt}(L)$ such that
  $p\psi(u)q\in\varphi(L)$, there exist $x,y\in B^*$ such that
  $\varphi(x)=p$ and $\varphi(y)=q$. Since $\varphi$ recognizes
  $\overline{L}$, we deduce that $x\hat{\alpha}(u)y\in\overline{L}$.
  By~(\ref{item:inequalities-for-syntactic-2}), it follows that
  $x\hat{\alpha}(v)y\in\overline{L}$, and so indeed
  $p\psi(v)q\in\varphi(L)$, as claimed.
\end{proof}

Lemma~\ref{l:inequalities-for-syntactic} serves to establish a simple
profinite characterization of when an inequality holds in a
pseudovariety of ordered monoids.

\begin{Prop}
  \label{p:order-vs-languages}
  Let \pv V be a pseudovariety of ordered monoids and let $u,v\in\Om
  AM$. The following conditions are equivalent:
  \begin{enumerate}
  \item\label{item:order-vs-languages-1} the inequality $u\le v$ holds in~\pv V;
  \item\label{item:order-vs-languages-2} whenever $L\subseteq A^*$ is
    a \pv V-recognizable language, $u\in\overline{L}$ implies
    $v\in\overline{L}$.
  \end{enumerate}
\end{Prop}

\begin{proof}
  $(\ref{item:order-vs-languages-1})
  \Rightarrow(\ref{item:order-vs-languages-2})$ %
  Suppose that (\ref{item:order-vs-languages-1}) holds and let
  $L\subseteq A^*$ be a \pv V-recognizable language such that
  $u\in\overline{L}$. By assumption, $\mathrm{Synt}(L)$ satisfies the
  inequality $u\le v$. By
  Lemma~\ref{l:inequalities-for-syntactic}(\ref{item:inequalities-for-syntactic-2}),
  taking $x=y=1$ and $\alpha$ to be the identity mapping, we deduce
  from $u\in\overline{L}$ that $v\in\overline{L}$, as required.

  $(\ref{item:order-vs-languages-2})
  \Rightarrow(\ref{item:order-vs-languages-1})$ %
  Since \pv V is generated by the syntactic ordered monoids
  $\mathrm{Synt}(K)$
  in~\pv V of regular languages $K$, it suffices to show that, for
  every \pv V-recognizable language $K\subseteq B^*$, the ordered
  monoid $\mathrm{Synt}(K)$ satisfies the inequality $u\le v$. For
  this purpose, we establish the condition
  (\ref{item:inequalities-for-syntactic-2}) of
  Lemma~\ref{l:inequalities-for-syntactic}. Thus, we should show that,
  for every homomorphism $\alpha \colon A^* \to B^*$ and all words
  $x,y\in B^*$, $x\hat{\alpha}(u)y\in\overline{K}$ implies
  $x\hat{\alpha}(v)y\in\overline{K}$.

  We claim that, for $w\in\Om AM$, $x\hat{\alpha}(w)y\in\overline{K}$ is
  equivalent to $w\in\overline{L}$, where
  $L=\alpha^{-1}(x^{-1}Ky^{-1})$. Suppose first that
  $x\hat{\alpha}(w)y\in\overline{K}$ and let $w_n$ be a sequence of words
  of~$A^*$ converging to~$w$. Since $\overline{K}$ is an open set such
  that $\overline{K}\cap B^*=K$
  (cf.~\cite[Theorem~3.6.1]{Almeida:1994a}) and the limit
  $x\hat{\alpha}(w)y$ of the sequence $x\alpha(w_n)y$ belongs
  to~$\overline{K}$, passing to a subsequence if necessary, we may
  assume that all terms in the sequence
  belong to~$K$. Then every $w_n$~belongs to $L$, so
  $w\in\overline{L}$. The converse follows from
  the continuity of both~$\hat{\alpha}$ and multiplication.

  It remains to apply the hypothesis (\ref{item:order-vs-languages-2})
  to the language $L=\alpha^{-1}(x^{-1}Ky^{-1})$, which is \pv
  V-recognizable because the class of all \pv V-recognizable languages
  constitutes a positive variety of languages.
\end{proof}

\section{Lifting factorizations}
\label{sec:lifting}

The next result shows that factorizations may be lifted along
inequalities. This property plays a key role in the sequel. The proof
uses nets in compact spaces, which is a classical tool in Topology
(see, for instance \cite{Willard:1970}).

\begin{Thm}
  \label{t:lifting-factorizations}
  Let \pv V be a polynomially closed pseudovariety of ordered monoids
  and let $u,v\in\Om AM$. If the inequality $u\le v$ holds in~\pv V
  then, for every factorization $u=u_0au_1$ with $a\in A$, there is a
  factorization $v=v_0av_1$ such that each inequality $u_i\le v_i$
  holds in~\pv V ($i=0,1$).
\end{Thm}

\begin{proof}
  Let $u=u_0au_1$ be an arbitrary factorization with $a\in A$. Let $I$
  be the set of all pairs $(L_0,L_1)$ of \pv V-recognizable languages
  $L_i\subseteq A^*$ such that $u_i\in\overline{L_i}$ ($i=0,1$). We
  consider on~$I$ the partial order defined by $(L_0,L_1)\le(K_0,K_1)$
  if $K_i\subseteq L_i$ ($i=0,1$). Since the positive variety of
  languages corresponding to~\pv V is closed under intersection, the
  above partial order on the set~$I$ is upper directed.
  For each $(L_0,L_1)\in I$, since $u=u_0au_1\in
  \overline{L_0}a\overline{L_1}=\overline{L_0aL_1}$ and the language
  $L_0aL_1$ is \pv V-recognizable because \pv V is polynomially
  closed, it follows from Proposition~\ref{p:order-vs-languages} that
  $v\in\overline{L_0}a\overline{L_1}$, and so we may choose a
  factorization $v=v_0^{(L_0,L_1)}av_1^{(L_0,L_1)}$ such that each
  $v_i^{(L_0,L_1)}$ belongs to~$\overline{L_i}$ ($i=0,1$).

  The mapping $\eta\colon I\to(\Om AM)^2$ defined by
  $\eta(L_0,L_1)=(v_0^{(L_0,L_1)},v_1^{(L_0,L_1)})$ may be viewed as a
  net in the compact product space $(\Om AM)^2$. Hence, there is a
  convergent subnet $\eta\circ\lambda$ determined by a mapping
  $\lambda\colon J\to I$ from another upper directed set $J$ into $I$ such
  that, for every $(L_0,L_1)\in I$ there is some $j\in J$ such that
  $(L_0,L_1)\le\lambda(j)$. Let $(v_0,v_1)$ be the limit of the subnet
  $\eta\circ\lambda$. By continuity of multiplication, since
  $v=v_0^{\lambda(j)}av_1^{\lambda(j)}$ for every $j\in J$, it follows
  that $v=v_0av_1$.

  We claim that each inequality $u_i\le v_i$ holds in~\pv V ($i=0,1$).
  To establish these inequalities, we apply again
  Proposition~\ref{p:order-vs-languages}: it suffices to show that,
  for all pairs of \pv V-recognizable languages $L_i\subseteq A^*$ such
  that $u_i\in\overline{L_i}$ ($i=0,1$), we have $v_i\in\overline{L_i}$
  ($i=0,1$). Let $j_0\in J$ be such that $(L_0,L_1)\le\lambda(j_0)$.
  Now, for every $i \in \{0,1\}$ and $j\in J$ such that $j\ge j_0$,
  the pseudoword
  $v_i^{\lambda(j)}$ belongs to the closure of the $i$th component of
  $\lambda(j)$, which is contained in $\overline{L_i}$. Hence
  $v_i^{\lambda_(j)}\in\overline{L_i}$ for every $j\ge j_0$, which
  implies that the limit $v_i$ also belongs to the closed
  set~$\overline{L_i}$, as claimed.
\end{proof}

Most of the time, it will be inconvenient to keep referring to the
letter $a$ in the factorizations to be lifted along inequalities
considered in Theorem~\ref{t:lifting-factorizations}. The following
result avoids it and further extends the lifting to an arbitrary
number of factors.

\begin{Cor}
  \label{c:lifting-factorizations}
  Let \pv V be a polynomially closed pseudovariety of ordered monoids
  and let $u,v\in\Om AM$. If the inequality $u\le v$ holds in~\pv V
  then, for every factorization $u=u_1\cdots u_n$, there is a
  factorization $v=v_1\cdots v_n$ such that each inequality $u_i\le
  v_i$ holds in~\pv V ($i=1,\ldots,n$).
\end{Cor}

\begin{proof}
  Proceeding by induction, it suffices to consider the case $n=2$. So,
  suppose that the inequality $u\le v$ holds in~\pv V and consider a
  factorization $u=u_1u_2$. If $u_2=1$, then the factorization
  $v=v\cdot1$ has the required properties. Otherwise, taking into
  account that $u_2$ is the limit of a sequence of nonempty words, a
  standard compactness argument shows that there is a factorization
  $u_2=au'$, where $a$ is some letter from~$A$. We may then apply
  Theorem~\ref{t:lifting-factorizations} to lift the factorization
  $u=u_1au'$ along the inequality $u\le v$ to a factorization
  $v=v_1av'$ such that the inequalities $u_1\le v_1$ and $u'\le v'$
  hold in~\pv V. It remains to take $v_2=av'$ and observe that the
  inequality $u_2\le v_2$ also holds in~\pv V.
\end{proof}

\section{Omega-inequalities}
\label{sec:om-ineq}

It is well known that elements of the free profinite monoid \Om AM may
be viewed as operations with a natural interpretation on each
profinite monoid in such a way that the interpretation is preserved
under continuous homomorphisms (see, for instance,
\cite{Almeida:2003cshort}). More precisely, each $w\in\Om AM$ defines
an $A$-ary operation symbol which is \emph{naturally interpreted} in a
profinite monoid $M$ as the operation $w_M \colon M^A\to M$ that maps each
function $\varphi\colon A\to M$ to~$\hat{\varphi}(w)$.
By an \emph{implicit signature}, we mean a set of such operation
symbols including the binary multiplication and the (nullary) identity
element, that is, the standard signature for working with monoids. For
an implicit signature $\sigma$, each profinite monoid $M$ has thus a
natural structure of a $\sigma$-algebra by interpreting each operation
symbol naturally. It is a simple exercise to show that, for a
pseudovariety \pv V of ordered monoids, the $\sigma$-subalgebra of~\Om
AV generated by~$A$, denoted by \oms AV, is a \pv V-free
$\sigma$-algebra on $A$. Elements of~\oms AV will be called
\emph{$\sigma$-words (over~\pv V)}.

The absolutely free $\sigma$-algebra on a generating set $A$ is the
term $\sigma$-algebra $T_\sigma(A)$. The members of $T_\sigma(A)$ are
obtained recursively from the elements of~$A$ by formally applying
successively an operation from $\sigma$. As is standard, they may be
visualized as finite rooted trees in which the leaves are labeled by
members of~$A$ or the constant $1$ and the non-leaf nodes are labeled
by elements of~$\sigma$; for each non-leaf node with operation $w$ the
sons are written in the order they are taken as arguments of the
operation $w$. A node is called a \emph{right descendant} whenever it
is the second son of a node labeled by a binary operation.

Such construction and representation are unique for each $\sigma$-term
$t$.
For a pseudovariety \pv V of ordered monoids, there is a \emph{natural}
homomorphism of $\sigma$-algebras $\theta\colon T_\sigma(A)\to\oms AV$
mapping each free generator to itself. For a $\sigma$-word $u$
over~\pv V, a $\sigma$-term in $\theta^{-1}(u)$ is said to
\emph{represent}~$u$.

We will be concerned with the signature $\omega$ consisting of binary
multiplication and the usual $\omega$-power, whose natural
interpretation on a profinite monoid $M$ maps each element $s$ to the
unique idempotent in the closed subsemigroup of~$M$ generated by~$s$.
As an example, consider the $\omega$-word $(a^2b^\omega)^\omega ab^\omega$.
One of its $\omega$-term representations is described by the 
tree in Figure~\ref{fig:tree}.
Note that the tree has four right descendants, which are graphically
indicated by the lower end of an edge going down from a node towards
the right.
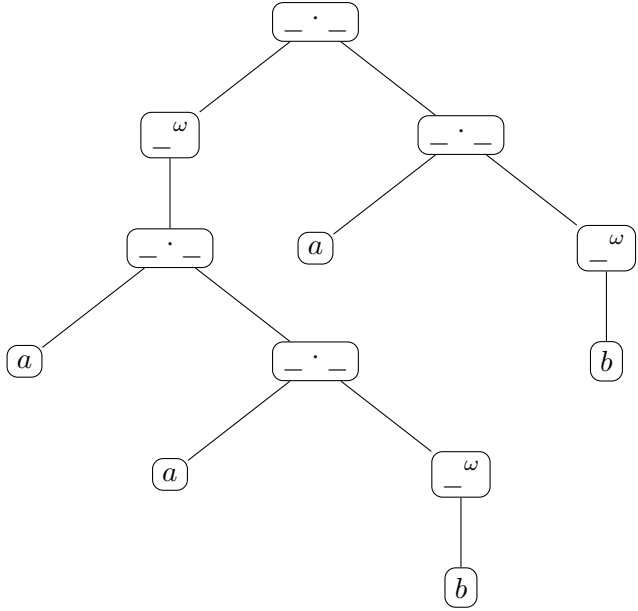
\begin{figure}[ht]
  \centering
  \begin{tikzpicture}[sibling distance=10em,
    every node/.style = {shape=rectangle, rounded corners,
      draw, align=center}]]
    \node {$\_\cdot\_$}
    child { node {$\_^\omega$}
      child { node {$\_\cdot\_$}
        child { node {$a$}}
        child { node {$\_\cdot\_$}
          child { node {$a$}}
          child { node {$\_^\omega$}
            child { node {$b$} }
          }
        }
      }
    }
    child { node {$\_\cdot\_$} 
      child { node {$a$} }
      child { node {$\_^\omega$}
        child { node {$b$} }
      }
    };
  \end{tikzpicture}  
  \caption{A tree representing the $\omega$-word $(a^2b^\omega)^\omega
    ab^\omega$.}
  \label{fig:tree}
\end{figure}

By the \emph{$\omega$-inequality problem} for a pseudovariety \pv V
of ordered monoids we mean the problem that takes
as input a pair $(u,v)$ of $\omega$-terms and asks whether
the inequality $u\le v$ is valid in~\pv V. Decidability of this
problem amounts to being able to algorithmically calculate in the
ordered monoid \omo AV.
Replacing inequalities by equalities, one may analogously define the
corresponding notions such as \emph{decidability} of the
\emph{$\omega$-equality problem}.

Our aim is to show that the $\omega$-inequality problem is
decidable for all levels of the Straubing-Th\'erien hierarchy of
aperiodic monoids. Rather than trying to construct efficient
algorithms for such an infinite class of problems, we concentrate on
``theoretical decidability'', proving simply that both the instances
of the problem with positive solution and those with negative solution
may be recursively enumerated.
For the purpose of enumerating inequalities with certain properties,
it is convenient to consider only finite alphabets contained in a
fixed countable set of variables, which we do from hereon without
further mention.

We say that a pseudovariety of ordered monoids \pv V is
\emph{recursively enumerable} if there is a Turing machine that
successively produces as outputs precisely the elements of~\pv V, up
to isomorphism, and nothing else. Equivalently, the corresponding
positive variety of languages is recursively enumerable.

\begin{Prop}
  \label{p:re-negative-cases}
  Let \pv V be a recursively enumerable pseudovariety of ordered
  monoids. Then the $\omega$-inequality problem for~\pv V is
  co-recursively enumerable.
\end{Prop}

\begin{proof}
  There is a Turing machine that successively enumerates all pairs
  consisting of a member $M$ of~\pv V, up to isomorphism, and an
  $\omega$-inequality $u\le v$, and nothing else. Note that the
  $\omega$-power of an element $s$ is computable in~$M$: for instance,
  one may compute it as $s^\omega=s^{n!}$, where $n=|M|$. Hence, one
  may effectively test for each such pair $(M,u\le v)$ whether $M$
  satisfies $u\le v$ and output the pair $u\le v$ in the negative
  case. This way, we recursively enumerate precisely the inputs for
  the $\omega$-inequality problem with negative output.
\end{proof}

We are thus left with enumerating the positive cases of the
$\omega$-inequality problem for suitable pseudovarieties of ordered
monoids~\pv V, a task that is accomplished in
Theorem~\ref{t:main-tool}.
We are actually going to deal with inequalities of $\omega$-words
from $\omo AA$, instead of inequalities of $\omega$-terms,
in order to be able to employ known properties of the monoids $\omo AA$.
With this aim, we generalize the $\omega$-inequality problem
from $\omega$-terms to $\omega$-words in a given pseudovariety.
Let \pv W be a pseudovariety of monoids with decidable $\omega$-equality
problem, that is, such that calculations in the ordered monoids \omo AW
can be performed algorithmically. By the \emph{$\omega$-inequality
problem over~\pv W} for a pseudovariety \pv V of ordered monoids
contained in~\pv W we mean the problem that takes as input a pair
$(u,v)$ of $\omega$-words from \omo AW and asks whether
the inequality $u\le v$ is valid in~\pv V.
The following result allows us to deal only with such restricted
$\omega$-inequality problems, when convenient.

\begin{Prop}
  \label{p:raising-ineq-prob}
  Let \pv W be a pseudovariety of monoids with decidable
  $\omega$-equality problem and let \pv V be a pseudovariety of
  ordered monoids contained in~\pv W such that the $\omega$-inequality
  problem for~\pv V over~\pv W is recursively enumerable. Then, the
  $\omega$-inequality problem for~\pv V is recursively enumerable.
\end{Prop}

\begin{proof}
  Let $\pi\colon \Om AM\to\Om AW$ be the unique continuous homomorphism
  mapping each generator to itself. By assumption, there is a Turing
  machine enumerating the quadruples $(u,v,w,z)$, where $u,v\in\omo AW$
  and $w,z\in\omo AM$ are such that $u\le v$ holds in~\pv V.
  Using the decidability of the $\omega$-equality problem for \pv W,
  for each such quadruple, algorithmically calculate $\pi(u)$ and $\pi(v)$,
  and check whether they are equal to $w$ and $z$, respectively.
  In the affirmative case, output the $\omega$-inequality $w\le z$.
  The resulting Turing machine recursively enumerates the positive cases
  of the $\omega$-inequality problem for~\pv V.
\end{proof}

\section{Equidivisibility and factoriality}
\label{sec:equidiv}

We recall here the notion of equidivisibility, which was introduced
in~\cite{McKnight&Storey:1969}. A semigroup $S$ is said to be
\emph{equidivisible} if, for all $s,t,u,v\in S$ such that $st=uv$,
there is $w\in S^1$ such that either $s=uw$ and $wt=v$, or $sw=u$ and
$t=wv$. A pseudovariety of monoids \pv V is \emph{equidivisible} if
\Om AV is equidivisible for every finite set $A$.
A pseudovariety of monoids \pv V is said to be \emph{closed under
  concatenation} if, for all \pv V-recognizable languages
$K,L\subseteq A^*$, the language $KL$ is also \pv V-recognizable.
Similar notions may be considered for pseudovarieties of semigroups.

As shown in~\cite[Lemma~2.3]{Almeida&ACosta:2007a} for pseudovarieties
of semigroups, such a pseudovariety \pv V containing all finite
nilpotent semigroups is closed under concatenation if and only if the
multiplication in~\Om AV is an open mapping for each finite set~$A$.
As mentioned in Section~\ref{sec:prelims}, the requirement that \pv V
contain all finite nilpotent semigroups is made to ensure that the
subsemigroup of~\Om AV generated by the free generators is free and
discrete. In the case of pseudovarieties of monoids, the corresponding
sufficient condition is that \pv V contain~\pv{MN}. Under such an
assumption for a pseudovariety of monoids \pv V, the argument in the
proof of~\cite[Lemma~2.3]{Almeida&ACosta:2007a} also yields that \pv V
is closed under concatenation if and only if the multiplication in~\Om
AV is an open mapping for each finite set~$A$.

Another relevant result from the same paper
is~\cite[Lemma~4.8]{Almeida&ACosta:2007a}, whose proof shows that, if
the multiplication in~\Om AV is an open mapping, then \Om AV is
equidivisible. In particular, since the pseudovariety \pv A is closed
under concatenation, the multiplication in each profinite aperiodic
monoid \Om AA is an open mapping and \pv A is equidivisible.

For an $\omega$-term $s$ and a positive integer $k$, the $\omega$-term
$s^k$ is defined recursively by $s^1 = s$ and $s^{k+1} = s^k \cdot s$.
The need for expliciting this definition is due to the fact that the
multiplication of $\omega$-terms is not associative, as we prefer to
keep track of the order in which the operations are performed.
Additionally, we assume that, for all $\omega$-terms $s$ and $t$, both
$s^0 \cdot t$ and $t \cdot s^0$ denote~$t$.

\emph{Decompositions} of an $\omega$-term $t$ are pairs of
$\omega$-terms, defined inductively with respect to the structure
of~$t$. If $t$ is either $1$ or a letter, then its decompositions are
$(1,t)$ and $(t,1)$. If $t = t_1 \cdot t_2$, then $t$ has
decompositions of two symmetric forms:
\begin{itemize}
\item for every decomposition $(s_1,s_2)$ of $t_1$, the pair $(s_1,
  s_2 \cdot t_2)$ is a decomposition of $t$;
\item for every decomposition $(s_1,s_2)$ of $t_2$, the pair $(t_1
  \cdot s_1, s_2)$ is a decomposition of $t$.
\end{itemize}
If $t = s^{\omega}$, then the decompositions of $t$ are defined to be
the pairs of the form $(s^k \cdot s_1, s_2 \cdot s^\ell)$,
where $(s_1, s_2)$ is a decomposition of $s$ and $k,\ell$ are
non-negative integers or~$\omega$, with at least one of them equal
to~$\omega$.

In order to illustrate the definition with a concrete example,
let $a,b$ be a pair of letters.
Decompositions of the term $a\cdot b$ are exactly
$(1, a\cdot b)$, $(a, 1\cdot b)$, $(a\cdot 1, b)$ and $(a\cdot b, 1)$.
Consequently, decompositions of the term $(a\cdot b)^\omega$ are
\begin{align*}
&((a\cdot b)^k\cdot 1, (a\cdot b)\cdot (a\cdot b)^\ell),
& &((a\cdot b)^k\cdot a, (1\cdot b)\cdot (a\cdot b)^\ell),\\
&((a\cdot b)^k\cdot (a\cdot 1), b\cdot (a\cdot b)^\ell),
& &((a\cdot b)^k\cdot  (a\cdot b), 1 \cdot (a\cdot b)^\ell),
\end{align*}
where $k$ and $\ell$ are non-negative integers or~$\omega$,
with at least one of them equal to~$\omega$.
In particular, the decompositions obtained for $\ell=0$ and $k=\omega$ are
\begin{align*}
&((a\cdot b)^\omega\cdot 1, a\cdot b ),
& &((a\cdot b)^\omega\cdot a, 1\cdot b ),\\
&((a\cdot b)^\omega\cdot (a\cdot 1), b ),
& &((a\cdot b)^\omega\cdot  (a\cdot b), 1 ).
\end{align*}
Note that $((a\cdot b)^\omega,1)$, $((a\cdot b)^\omega\cdot a, b )$
and $((a\cdot b)^\omega, (a\cdot b)^\omega)$
fail to be decompositions of $(a\cdot b)^\omega$.
However, as the next result shows, all factorizations of $\omega$-words
may be obtained from decompositions of $\omega$-terms representing
them. It appears to be intimately related with results
of~\cite{Almeida&Costa&Zeitoun:2015a}, but not quite to follow from
them in any direct way. We thus include here a direct proof. The key
ingredients are the facts that \Om AA is equidivisible and its
multiplication is an open mapping.

\begin{Thm}
  \label{t:omega-term-decomposition}
  Let $t$ be an $\omega$-term and $u, v \in \Om AA$. Then, $t$
  represents the product $u v$ if and only if there exists a
  decomposition $(t_1,t_2)$ of $t$ such that $t_1$ represents $u$ and
  $t_2$ represents $v$.
\end{Thm}

\begin{proof}
  That a decomposition of~$t$ yields a factorization of the
  $\omega$-word in~\omo AA represented by~$t$ amounts to a
  straightforward induction on the construction of~$t$ as an
  $\omega$-term. For the converse, we also proceed by induction on the
  construction of $t$ as an $\omega$-term. At the basis of the
  induction are the identity term~$1$
  and the letters, for which the result is obvious: indeed, both 1 and
  the letters admit no nontrivial factorizations in $\Om AA$ and this
  corresponds exactly to the decompositions allowed in these cases.
    
  Suppose that $t=s_1 \cdot s_2$ is obtained by multiplying two (shorter)
  terms and let $w_i$ be the $\omega$-word in~\omo AA represented by
  $s_i$ ($i=1,2$). From the equality $w_1w_2=uv$, by equidivisibility
  of~\Om AA we deduce that there is some pseudoword $z \in \Om AA$ such that
  either $w_1=uz$ and $zw_2=v$, or $w_1z=u$ and $w_2=zv$. In the first
  case, applying the induction hypothesis to $s_1$, we obtain a
  decomposition $(s_{1,1},s_{1,2})$ of $s_1$ such that $s_{1,1}$ represents
  $u$ and $s_{1,2}$ represents~$z$. Hence, the $\omega$-terms $t_1=s_{1,1}$
  and $t_2=s_{1,2}\cdot s_2$ represent $u$ and~$v$, respectively. The
  other case is similar.

  Suppose next that $t=s^\omega$ where $s$ is a (shorter)
  $\omega$-term. In view of the preceding case, we may assume
  inductively that the result holds for each $\omega$-term $s^n$ with
  $n$ a positive integer. Let $w$ be the $\omega$-word in~\omo AA
  represented by~$s$. Since the sequence $(w^n)_n$ converges to the
  product $uv$ and the multiplication in~\Om AA is open, there is a
  strictly increasing sequence of indices $(n_k)_k$ such that there is
  a factorization $w^{n_k}=u_kv_k$ with $\lim u_k=u$ and $\lim v_k=v$
  (cf.~\cite[Lemma~2.5]{Almeida&ACosta:2007a}). By the induction
  hypothesis applied to $s^{n_k}$, for each $k$ there is a decomposition
  $(t_{k,1},t_{k,2})$ of~$s^{n_k}$ such that $t_{k,1}$ represents
  $u_k$ and $t_{k,2}$ represents~$v_k$. Moreover, by the definition of
  decomposition and a simple induction, for each $k$ there exists
  $\ell_k$, with $0 \leq \ell_k \leq n_k - 1$, such that
  $t_{k,1}=s^{\ell_k}\cdot t'_{k,1}$ and
  $t_{k,2} = (\dots ((t'_{k,2} \cdot s) \cdot s) \dots ) \cdot s$, where
  $(t'_{k,1},t'_{k,2})$ is a decomposition of~$s$ and in the latter term
  the number of multiplications by $s$ is $n_k-\ell_k-1$. Let $w_{k,i}$ be
  the $\omega$-word in~\omo AA represented by $t'_{k,i}$ ($i=1,2$).
  Then the equalities $u_k = w^{\ell_k}\cdot w_{k,1}$ and
  $v_k = w_{k,2} \cdot w^{n_k-\ell_k-1}$ hold in $\omo AA$,
  and the decomposition of $s$ yields the factorization
  $w = w_{k,1} w_{k,2}$. Up to taking a subsequence of $(n_k)_k$,
  we may assume that each of the
  sequences $(w_{k,i})_k$ converges to some $w_i$ ($i=1,2$) and that
  each of the sequences $(\ell_k)_k$ and $(n_k-\ell_k-1)_k$ is either
  constant or strictly increasing.
  Let $\ell$ denote the common value of $\ell_k$ if the sequence
  $(\ell_k)_k$ is constant, and $\omega$ if the sequence is strictly increasing.
  Similarly, let $m$ denote either the common value of $n_k-\ell_k-1$
  or $\omega$. Continuity of
  multiplication in~\Om AA yields the equalities $u=w^\ell w_1$,
  $v=w_2w^m$, and $w=w_1w_2$. The latter equality and the induction
  hypothesis applied to the $\omega$-term $s$ provide a decomposition
  $(s_1,s_2)$ of $s$ such that $s_i$ represents $w_i$ ($i=1,2$). It remains
  to take $t_1=s^\ell \cdot s_1$ and $t_2=s_2 \cdot s^m$ to obtain
  the required decomposition.
\end{proof}

A subset $X$ of a semigroup $S$ is said to be \emph{factorial in~$S$}
if $st\in X$ with $s,t\in S$ implies $s,t\in X$. For an implicit
signature $\sigma$, a pseudovariety \pv V is said to be
\emph{$\sigma$-factorial} if \oms AV is factorial in~\Om AV for every
finite set~$A$.

The following result is an immediate application of
Theorem~\ref{t:omega-term-decomposition}. It was first proved in
\cite{Almeida&Costa&Zeitoun:2015} as a by-product of a
language-theoretical proof of the correctness of McCammond's algorithm
for solving the $\omega$-identity problem for~\pv A
\cite{McCammond:1999a}. An alternative proof and
a generalization to all Burnside pseudovarieties
$B_n=\op x^{\omega+n}=x^\omega\cl$ can be found
in~\cite{Almeida&Costa&Zeitoun:2015a}.

\begin{Cor}
  \label{c:factoriality}
  The pseudovariety \pv A is $\omega$-factorial.\qed
\end{Cor}

Combining Corollary~\ref{c:factoriality} with
Corollary~\ref{c:lifting-factorizations}, we obtain the following result.

\begin{Cor}
  \label{c:lifting-omega-factorizations}
  Let \pv V be a polynomially closed pseudovariety of aperiodic
  ordered monoids and let $u,v\in\omo AA$. If the inequality $u\le v$
  holds in~\pv V then, for every factorization $u=u_1\cdots u_n$,
  there is a factorization $v=v_1\cdots v_n$ into $\omega$-words such
  that each inequality $u_i\le v_i$ holds in~\pv V
  ($i=1,\ldots,n$).\qed
\end{Cor}

\section{Repetitions in \texorpdfstring{$\omega$}{omega}-words}
\label{sec:repetitions}

This section is devoted to establishing another reflection of the
intuitively expected phenomenon that there are not many ways to
factorize an $\omega$-word. A precise formulation of how such
factorizations are obtained is already given by
Theorem~\ref{t:omega-term-decomposition}. For application in the next
section, we also need the following repetition result for iterated
factorizations.

\begin{Prop}
  \label{p:repetitions}
  Suppose, for each $i \geq 0$, $v_i, x_i$ are $\omega$-words in~\omo
  AA such that $v_i = x_i v_{i+1}$. Then there exist indices $i$ and
  $j$ with $i \leq j$ such that $v_i = (x_i\cdots x_j)^\omega v_{j+1}$.
\end{Prop}

As an example, take $v_0=(a^2b^\omega)^\omega ab^\omega$ and consider
the factorizations $v_i=x_iv_{i+1}$ given by the $\omega$-words
$v_i=b^i v_0$ and $x_i=b^i(a^2b^\omega)^\omega$ ($i\ge0$). Note that
there are infinitely many ways to factorize $v_0$. However,
Proposition 6.1 says that some repetition is always possible. In this
concrete example, this is quite easy, because one can take $i=j$ an
arbitrary index. In the general case we need to be more careful to
choose appropriate indices and we need to overcome certain technical
obstacles.

In order to prove this key proposition, we are going to study, instead
of the factorizations $v_i = x_i v_{i+1}$, the corresponding
syntactic decompositions of an $\omega$-term representing $v_0$,
with the aim of finding certain repetitions, which would allow to
repeat several consecutive $\omega$-terms in the decomposition
without changing the resulting $\omega$-word $v_i = x_i\cdots x_j
v_{j+1}$.

\begin{Lemma}
  \label{l:compatible-decompositions}
  Let $v \in \omo AA$ be an $\omega$-word, and for each $i \geq 1$
  let $v = v_i \tilde{v}_i$ be a factorization in~\omo AA.
  Then there exist indices $i$ and $j$ with $i < j$ such that
  $v = v_j \tilde{v}_i$.
\end{Lemma}

\begin{proof}
  Let $t$ be an $\omega$-term representing~$v$. We proceed by
  induction on the construction of~$t$. If $t$ is a letter from $A$,
  then the only factorizations of $v$ are $1 \cdot v$ and $v \cdot 1$
  by Theorem~\ref{t:omega-term-decomposition}, and the claim is
  obvious.

  In case $t = t_1 \cdot t_2$, denote by $u_1$ and $u_2$ the $\omega$-words
  represented by $t_1$ and $t_2$, respectively. By equidivisibility,
  for each $i$ there exists $y_i \in \omo AA$ such that either $u_1 =
  v_i y_i$ and $y_i u_2=\tilde{v}_i$, or $u_1 y_i=v_i$ and $u_2 = y_i
  \tilde{v}_i$. By symmetry, we may assume that the former case occurs
  infinitely often. Applying the induction hypothesis to the
  factorizations of $u_1$, we obtain indices $i$ and $j$ with $i < j$
  such that $u_1 = v_j y_i$. It follows that $v = u_1 u_2 = v_j y_i
  u_2 = v_j \tilde{v}_i$, as required.

  Finally, consider the case $t = s^{\omega}$. According to
  Theorem~\ref{t:omega-term-decomposition}, for each $i$ there exists
  a decomposition $(s_{i,1},s_{i,2})$ of the $\omega$-term $s$ and
  exponents $k_i,\ell_i$, that are either non-negative integers
  or~$\omega$, such that $v_i$ is represented by the $\omega$-term
  $s^{k_i} \cdot s_{i,1}$ and $\tilde{v}_i$ is represented by the
  $\omega$-term $s_{i,2} \cdot s^{\ell_i}$. Since for each $i$ at
  least one of $k_i$ and $\ell_i$ is equal to $\omega$, without loss
  of generality we may assume that $\ell_i = \omega$ for all $i$.
  Consider the $\omega$-words $w_i$ and $\tilde{w}_i$ represented by
  the $\omega$-terms $s_{i,1}$ and $s_{i,2}$, respectively. Then, all
  products $w_i \tilde{w}_i$ are equal to the $\omega$-word $u$
  represented by $s$. Using the induction hypothesis, we obtain
  indices $i$ and $j$ with $i < j$ such that $u = w_j \tilde{w}_i$.
  Then $v = u^{\omega} = u^{k_j} w_j \tilde{w}_i u^{\omega} = v_j
  \tilde{v}_i$, which concludes the proof.
\end{proof}

For the proof of Proposition~\ref{p:repetitions}, it is convenient to
introduce a measure of depth of $\omega$-terms. With this aim, let $t$
be an arbitrary $\omega$-term.
In terms of the tree of~$t$, we consider two parameters:
\begin{itemize}
\item the maximum number $\mu_\omega(t)$ of occurrences of the
  $\omega$-power in a branch of the tree of~$t$;
\item for each branch of the tree of~$t$ with $\mu_\omega(t)$
  occurrences of the $\omega$-power, let $n$ be the number of right
  descendants above the top occurrence of the $\omega$-power in the
  branch; 
  the number $\mu_\ell(t)$ is the maximum of all such $n$
  (if $\mu_\omega(t) = 0$, then $\mu_\ell(t)$ is equal to
  the maximum number of right descendants in a branch).
\end{itemize}
Let $\mu(t)$ stand for the pair $(\mu_{\omega}(t),\mu_{\ell}(t))$.
For example, for the term $t$ depicted in
Figure~\ref{fig:tree}, we have $\mu(t)=(2,0)$, while if the term $s$ is
obtained by replacing in~$t$ the rightmost leaf $b$ by the term
$b^\omega$, then $\mu(s)=(2,1)$.

Using the lexicographic ordering $\leq$ of pairs of non-negative
integers, the definition of $\mu$
can be rephrased inductively with respect to the structure of terms
as follows:
$\mu(1) = \mu(a) = (0,0)$ for all $a \in A$; $\mu(s^{\omega})
= (\mu_{\omega}(s)+1,0)$; $\mu(t_1 \cdot s^{\omega}) =
\max\{\mu(t_1),\mu(s^{\omega})\}$; $\mu(t_1 \cdot t_2) =
\max\{\mu(t_1),\mu(t_2)+(0,1)\}$ if $t_2$ is not of the form $s^{\omega}$,
where the addition of pairs is taken component-wise.

\begin{Lemma}
  \label{l:mu-vs-decompositions}
  If $t$ is an $\omega$-term and $(\bar{t},t')$ is a decomposition
  of~$t$, then $\mu(t)\ge\mu(t')$ and the equality can only occur in
  the following four cases:
  \begin{enumerate}
    \item \label{item:mu-vs-decompositions-1}$t$ is $1$ or a letter;
    \item \label{item:mu-vs-decompositions-2}$t = s_1 \cdot s_2$ and
      $(\bar{t},t') = (\bar{s},s' \cdot s_2)$ for some decomposition
      $(\bar{s},s')$ of $s_1$;
    \item \label{item:mu-vs-decompositions-3}$t = s^{\omega}$ and
      $(\bar{t},t') = (s^k \cdot \bar{s},s' \cdot s^{\omega})$ for
      some decomposition $(\bar{s},s')$ of~$s$ and
      $k \in \{0,1,2,\ldots\} \cup \{\omega\}$;
    \item \label{item:mu-vs-decompositions-4}$t = s_1 \cdot
      s^{\omega}$ and $(\bar{t},t') =
      (s_1 \cdot(s^k \cdot \bar{s}),s' \cdot s^{\omega})$ for
      some decomposition $(\bar{s},s')$ of~$s$ and
      $k \in \{0,1,2,\ldots\} \cup \{\omega\}$.
  \end{enumerate}
\end{Lemma}

\begin{proof}
  We proceed by induction on the construction of the term~$t$.
  The claim is obvious for $t$ equal to $1$ or a letter, which means
  that we fall in Case~(\ref{item:mu-vs-decompositions-1}).

  Suppose that $t = s_1 \cdot s_2$ and $(\bar{t},t') = (\bar{s},s'
  \cdot s_2)$, with $(\bar{s},s')$ a decomposition of $s_1$. Since
  these assumptions fall in Case~(\ref{item:mu-vs-decompositions-2}),
  it is enough to prove the inequality $\mu(t)\ge\mu(t')$. We
  distinguish two cases. For $s_2$ of the form $s^{\omega}$ we get
  \begin{equation*}
  \mu(t) = \max\{\mu(s_1),\mu(s^\omega)\} \geq
  \max\{\mu(s'),\mu(s^\omega)\} = \mu(s' \cdot s_2),
  \end{equation*}
  using the induction hypothesis on the $\omega$-term $s_1$.
  Similarly, if $s_2$ is not of the form $s^{\omega}$, we obtain
  \begin{equation*}
  \mu(t) = \max\{\mu(s_1),\mu(s_2)+(0,1)\} \geq
  \max\{\mu(s'),\mu(s_2)+(0,1)\} %
  = \mu(s' \cdot s_2).
  \end{equation*}
  If $t = s_1 \cdot s_2$ and $(\bar{t},t') = (s_1 \cdot \bar{s},s')$,
  with $(\bar{s},s')$ a decomposition of $s_2$, then we also
  distinguish two cases. For $s_2$ of the form $s^{\omega}$ we get
  \begin{equation*}
  \mu(t) = \max\{\mu(s_1),\mu(s^\omega)\} \geq
  \mu(s^\omega) = \mu(s_2) \geq \mu(s'),
  \end{equation*}
  using the induction hypothesis on the $\omega$-term $s_2$; in fact,
  the induction hypothesis yields that the last inequality is strict
  unless $(\bar{s},s')$ is of the form $(s^k \cdot \tilde{s},s'' \cdot
  s^{\omega})$ for some decomposition $(\tilde{s},s'')$ of~$s$ and $k
  \in \{0,1,2,\ldots\} \cup \{\omega\}$, which
  shows that the equality $\mu(t)=\mu(s')$ only holds in
  Case~(\ref{item:mu-vs-decompositions-4}).
  In the second case, assuming that $s_2$ is not of the form
  $s^{\omega}$, we obtain strict inequality:
  \begin{equation*}
  \mu(t) = \max\{\mu(s_1),\mu(s_2)+(0,1)\} > \mu(s_2) \geq \mu(s').
  \end{equation*}

  If $t = s^{\omega}$ and the decomposition $(\bar{t},t')$ of $t$ has
  $t' = s' \cdot s^{\omega}$, with $(\bar{s},s')$ a decomposition of
  $s$, meaning that Case~(\ref{item:mu-vs-decompositions-3}) holds, then
  \begin{equation*}
  \mu(s' \cdot s^{\omega}) = \max\{\mu(s'),\mu(s^\omega)\} =
  \mu(s^{\omega}),
  \end{equation*}
  because $\mu(s') \leq \mu(s) < \mu(s^\omega)$ by the induction
  hypothesis. Finally, if $t = s^{\omega}$ and the decomposition
  $(\bar{t},t')$ of $t$ has $t' = s' \cdot s^k$, with
  $(\bar{s},s')$ a decomposition of $s$ and $k$ a non-negative
  integer, then 
  \begin{equation*}
  \mu(s' \cdot s^k) \leq
  \max\{\mu(s'),(\mu_{\omega}(s),\mu_{\ell}(s)+2)\} < \mu(s^{\omega}),
  \end{equation*}
  because $\mu(s') \leq \mu(s) < \mu(s^{\omega})$ by the induction
  hypothesis, while a simple induction taking into account that
  $s^{k+1}=s^k \cdot s$ shows that the equality $\mu(s^k) = \mu(s) +
  (0,1)$ holds whenever $k>1$ in case $s$ is not of the form
  $s=r^\omega$, and $\mu(s^k)=\mu(s)$ otherwise.
\end{proof}

We may now proceed with the proof of the main result of this section.

\begin{proof}[Proof of Proposition~\ref{p:repetitions}]
  Let $t_0$ be an $\omega$-term representing~$v_0$. Using
  Theorem~\ref{t:omega-term-decomposition}, choose successively for
  each $i$ some decomposition $(s_i,t_{i+1})$ of the $\omega$-term
  $t_i$ such that $s_i$ represents $x_i$ and $t_{i+1}$ represents
  $v_{i+1}$.
  
  The proof of the proposition proceeds by induction with respect to
  $\mu(t_0)$, and for $\omega$-terms $t_0$ with the same value of
  $\mu(t_0)$ by induction with respect to the size of the term. By
  Lemma~\ref{l:mu-vs-decompositions}, we know that $\mu(t_{i+1}) \leq
  \mu(t_i)$ for every $i \geq 0$. If $\mu(t_i) < \mu(t_0)$ for some $i
  \geq 1$, then the
  statement follows directly from the induction assumption applied to
  $t_i$. Thus, assume that $\mu(t_i) = \mu(t_0)$ for all $i$, which
  implies that all decompositions $(s_i, t_{i+1})$ are of one of the
  four forms described in Lemma~\ref{l:mu-vs-decompositions}.

  If some term $t_i$ is $1$ or a letter, then the same is true for all
  terms $t_i$, and all terms $s_i$, except for at most one, are equal
  to $1$; then the statement is obviously true. It remains to deal
  with the situation when every decomposition $(s_i,t_{i+1})$ is of
  one of the forms \eqref{item:mu-vs-decompositions-2} to
  \eqref{item:mu-vs-decompositions-4}, which is assumed for the
  remainder of the proof. Let us first formulate a useful
  observation.

  \begin{claim}
    Assume that $i$ is such that the decomposition $(s_i, t_{i+1})$ of
    $t_i$ is of the form \eqref{item:mu-vs-decompositions-3} with $t_i
    = s^{\omega}$ or \eqref{item:mu-vs-decompositions-4} with $t_i =
    \bar{t}_i \cdot s^{\omega}$ for some $\omega$\nobreakdash-terms
    $s$ and $\bar{t}_i$.
    Then, for every $j > i$, the decomposition $(s_j,t_{j+1})$
    is of the form \eqref{item:mu-vs-decompositions-2}
    or~\eqref{item:mu-vs-decompositions-4}, 
    and there exist $\omega$\nobreakdash-terms $\bar{t}_j$ for $j
    > i$ such that $t_j = \bar{t}_j \cdot s^{\omega}$. Moreover, the
    inequality $\mu(\bar{t}_{i+1}) < \mu(t_{i+1})$ holds.
  \end{claim}

  \begin{proof}[Proof of Claim]
    We prove the statement that $t_j = \bar{t}_j \cdot s^{\omega}$ for
    some $\omega$\nobreakdash-terms $\bar{t}_j$ by induction with respect
    to $j>i$. Note that since $\mu(t_{j+1}) = \mu(t_j)$,
    this form of $t_j$ ensures that, for $j> i$,
    the decomposition $(s_j,t_{j+1})$ of $t_j$ is of the form
    \eqref{item:mu-vs-decompositions-2}
    or~\eqref{item:mu-vs-decompositions-4}.

    By the assumption that the decomposition $(s_i, t_{i+1})$ of $t_i$
    is of the form \eqref{item:mu-vs-decompositions-3} or
    \eqref{item:mu-vs-decompositions-4}, we know that the
    $\omega$-term $t_{i+1}$ has the form $\bar{t}_{i+1} \cdot
    s^\omega$ where $\bar{t}_{i+1}=s'$ in the notation of
    Lemma~\ref{l:mu-vs-decompositions}. Thus, we have proved the base
    of the induction $j=i+1$. Moreover, in both cases the term $s'$ is
    the second component of a decomposition $(\bar{s},s')$ of $s$
    and, therefore, $\mu(\bar{t}_{i+1}) = \mu(s') \leq \mu(s)$ by
    Lemma~\ref{l:mu-vs-decompositions}. Hence we get the last
    statement of the claim, because $\mu(s) < \mu(s^{\omega}) \leq
    \mu(t_{i+1})$.

    To prove the induction step, we first assume that the
    decomposition $(s_j,t_{j+1})$ of $t_j=\bar{t}_j \cdot s^{\omega}$,
    with $j>i$, is of the form \eqref{item:mu-vs-decompositions-2}.
    Then, we have $t_{j+1}=t'=s'\cdot s^\omega$ in the notation of
    case~\eqref{item:mu-vs-decompositions-2} of
    Lemma~\ref{l:mu-vs-decompositions}, where $s^\omega=s_2$.
    Hence, we may take $\bar{t}_{j+1}=s'$ and we are done in this
    case. We may proceed in the same way when we assume that the
    decomposition $(s_j,t_{j+1})$ of $t_j=\bar{t}_j \cdot s^{\omega}$,
    with $j>i$, is of the form \eqref{item:mu-vs-decompositions-4}.
\end{proof}
  
  We distinguish two cases for which the proof proceeds in different
  ways. First, assume that there exists an index $k$ such that for all
  $j \geq k$ the decomposition $(s_j,t_{j+1})$ of $t_j$ is of the
  form~\eqref{item:mu-vs-decompositions-2}, that is, there exist
  $\omega$\nobreakdash-terms $t'$ and $\bar{t}_j$ for $j \geq k$ such
  that $t_j = \bar{t}_j \cdot t'$ and $(s_j, \bar{t}_{j+1})$ is a
  decomposition of $\bar{t}_j$. Choose $k$ to be the least of such
  indices. Denote by $\bar{v}_j$ the $\omega$-word represented
  by~$\bar{t}_j$. Our goal is to apply the induction hypothesis to the
  sequence of factorizations $\bar{v}_j = x_j \bar{v}_{j+1}$ for $j
  \geq k$. This would be possible, once we verify either that
  $\mu(\bar{t}_k) < \mu(t_0)$ or that $\mu(\bar{t}_k) = \mu(t_0)$ and
  the size of $\bar{t}_k$ is smaller than that of~$t_0$. At least one
  of these conditions is certainly true if $k = 0$, since then
  $\bar{t}_k$ is a proper subterm of $t_0$. If $k \geq 1$, then the
  minimality of choice of $k$ ensures that the decomposition $(s_{k-1},
  t_k)$ of $t_{k-1}$ is of the form
  \eqref{item:mu-vs-decompositions-3}
  or~\eqref{item:mu-vs-decompositions-4}. Hence, we may apply the
  Claim with $i = k-1$ to deduce that $\mu(\bar{t}_k) < \mu(t_k)$.
  Since we have also assumed that $\mu(t_k) = \mu(t_0)$, the induction
  hypothesis now provides us with indices $i$ and $j$, with $i < j$,
  such that $\bar{v}_i = (x_i\cdots x_j)^\omega \bar{v}_{j+1}$.
  Multiplication by the $\omega$-word $w$ represented by the
  $\omega$-term $t'$ turns this equality into the required repetition
  $v_i = \bar{v}_i w = (x_i\cdots x_j)^\omega \bar{v}_{j+1} w =
  (x_i\cdots x_j)^\omega v_{j+1}$.

  Second, assume that such an index $k$ does not exist, so that,
  according to the claim,
  for infinitely many indices $i$ the decomposition $(s_i, t_{i+1})$
  is of the form~\eqref{item:mu-vs-decompositions-4}.
  Let $i_1 < i_2 < \cdots$ be all such indices.
  Then for every $m \geq 1$, all decompositions with indices between
  $i_m$ and $i_{m+1}$ are of the form~\eqref{item:mu-vs-decompositions-2}.
  For all $i \geq i_1$, the term $t_i$ can be written as
  $t_i = \bar{t}_i \cdot s^{\omega}$, for a fixed $\omega$-term $s$
  and certain $\omega$-terms $\bar{t}_i$.
  If $i$ is not one of the indices $i_m$, then $(s_i, \bar{t}_{i+1})$
  is a decomposition of $\bar{t}_i$.
  On the other hand, for all indices $i_m$ we have
  $s_{i_m} = \bar{t}_{i_m} \cdot (s^{k_m} \cdot \bar{s}_m)$,
  with $(\bar{s}_m, \bar{t}_{i_m+1})$ a decomposition of $s$
  and $k_m \in \{0,1,2,\ldots\} \cup \{\omega\}$.
  Denoting $\omega$-words in $\omo AA$ represented by $\omega$-terms
  $\bar{t}_i$, $s$ and $\bar{s}_m$ by $\bar{v}_i$, $w$ and $\bar{w}_m$,
  respectively, the above relations between these $\omega$-terms
  translate into the following equalities of $\omega$-words:
  $v_i = \bar{v}_i w^{\omega}$ for all $i \geq i_1$,
  $\bar{v}_i = x_i \bar{v}_{i+1}$ for indices $i$ other than~$i_m$,
  $x_{i_m} = \bar{v}_{i_m} w^{k_m} \bar{w}_m$, and
  $w = \bar{w}_m \bar{v}_{i_m+1}$.
  By Lemma~\ref{l:compatible-decompositions}, there exist indices
  $\ell,n \geq 1$, with $\ell < n$, such that
  $w = \bar{w}_n \bar{v}_{i_{\ell}+1}$.
  We are going to verify that $i_{\ell}+1$ and $i_n$ are the required
  indices $i$ and $j$.
  First we calculate the product of $\omega$-terms $x_i$
  between two consecutive indices $i_m$ as
  \begin{equation*}
  x_{i_m+1} \cdots x_{i_{m+1}} = x_{i_m+1} \cdots x_{i_{m+1}-1}
  \bar{v}_{i_{m+1}} w^{k_{m+1}} \bar{w}_{m+1} =
  \bar{v}_{i_m+1} w^{k_{m+1}} \bar{w}_{m+1},
  \end{equation*}
  using successively the equalities $\bar{v}_i = x_i \bar{v}_{i+1}$ for
  $i = i_{m+1}-1, \dots, i_m+1$.
  Then, we obtain
  $$x_{i_{\ell}+1} \cdots x_{i_n} =
  \bar{v}_{i_{\ell}+1} w^{k_{\ell+1} + \dots + k_n + n - \ell - 1} \bar{w}_n,$$
  using all equalities $w = \bar{w}_m \bar{v}_{i_m+1}$
  for $m = \ell+1, \dots, n-1$.
  Denoting the number $k_{\ell+1} + \dots + k_n + n - \ell - 1$ by $q$,
  we finally calculate
  \begin{align*}
    (x_{i_{\ell}+1} \cdots x_{i_n})^{\omega} v_{i_n+1}
    &= (\bar{v}_{i_{\ell}+1} w^q \bar{w}_n)^{\omega} \bar{v}_{i_n+1} w^{\omega}
    \\
    &= \bar{v}_{i_{\ell}+1} (w^q \bar{w}_n \bar{v}_{i_{\ell}+1})^{\omega}
      w^q \bar{w}_n \bar{v}_{i_n+1} w^{\omega}
    \\
    &= \bar{v}_{i_{\ell}+1} w^{\omega}
    = v_{i_{\ell}+1},
  \end{align*}
  using the equalities $\bar{w}_n \bar{v}_{i_{\ell}+1} =
  \bar{w}_n \bar{v}_{i_n+1} = w$.
\end{proof}

\section{Syntactic proofs of
  \texorpdfstring{$\omega$}{omega}-inequalities}
\label{sec:re-positive-cases}

The aim of this section is to show that the positive cases of the
$\omega$-inequality problem for an aperiodic pseudovariety of ordered
monoids of the form $\pv W=\PolB V$ may be derived from the positive
cases of the $\omega$-inequality problem for~\pv V. Since we deal only
with aperiodic monoids, it is convenient to consider only
$\omega$-inequality problems over~\pv A. By
Proposition~\ref{p:raising-ineq-prob}, the general $\omega$-inequality
problem reduces to this one since the $\omega$-equality problem
for~\pv A is decidable as was shown by McCammond
\cite{McCammond:1999a}. For this reason, all $\omega$-inequalities
considered in this section are taken over~\pv A.

In order to enumerate the positive cases of the $\omega$-inequality
problem for~\pv W, we show that, under suitable assumptions on~\pv V,
if an inequality $u\le v$ of $\omega$-words is valid in~\pv W, then
there is a finite syntactic proof of this fact from a convenient set
of hypotheses, or axioms. This is thus a sort of completeness result
for the choice of the set of axioms. Provided the axioms can be
enumerated, since formal proofs can then be enumerated, so can be the
provable facts.

To fulfill the above program, we need to make precise what kind of
formal proofs we will be considering. By a \emph{formal proof} of an
$\omega$-inequality $u\le v$ from a given set $\Gamma$ of hypotheses,
we mean a finite sequence $u_i\le v_i$ ($i=1,\ldots,n$) of
$\omega$-inequalities such that each member of the sequence satisfies
one of the following conditions:
\begin{itemize}
\item $u_i\le v_i$ is a member of~$\Gamma$;
\item $u_i\le v_i$ is of the form $u_ju_k\le v_jv_k$ with $j,k<i$;
\item $u_i\le v_i$ is of the form $u_j^\omega\le v_j^\omega$ with
  $j<i$;
\item there are $j,k<i$ such that $u_i=u_j$, $v_j=u_k$, and $v_k=v_i$.
\end{itemize}
If there is such a formal proof, then we say that $u\le v$ is provable
from~$\Gamma$ and we write $\Gamma\vdash u\le v$. Note that we do not
include in our proof rules the possibility of making substitutions of
variables by $\omega$-words. We will not need them because we will
consider sets of hypotheses that are closed under such substitutions.
Similarly, we do not need to take into account the possibility of
multiplying both sides of an inequality on the left and on the right
by the same $\omega$-words because we will consider sets of hypotheses
containing all trivial $\omega$-inequalities.

\begin{Prop}
  \label{p:provability}
  Let \pv V be a polynomially closed pseudovariety of aperiodic
  ordered monoids and let $\Gamma$ be the set of all trivial
  $\omega$-inequalities together with all inequalities of the form
  $u^\omega\le u^\omega vu^\omega$, with $u,v \in \omo AA$,
  such that the $\omega$-inequality
  $v\le u$ is valid in~\pv V. If an $\omega$-inequality $u\le v$ is
  valid in $\pv W=\PolB V$, then $\Gamma\vdash u\le v$.
\end{Prop}

\begin{proof}
  Let $t$ be an $\omega$-term representing~$u$. We proceed by
  induction on the construction of~$t$. If $t$ is a
  letter or $1$, then $u=v$,
  and so $u=v$ belongs to~$\Gamma$.

  In case $t=t_1 \cdot t_2$, we obtain the corresponding factorization
  $u=u_1u_2$, where $t_i$ represents $u_i$ ($i=1,2$). By
  Corollary~\ref{c:lifting-omega-factorizations}, there is a
  factorization $v=v_1v_2$ in~\omo AA such that the inequality $u_i\le
  v_i$ is valid in~\pv W. The induction hypothesis yields
  $\Gamma\vdash u_i\le v_i$ ($i=1,2$) and so also $\Gamma\vdash u\le
  v$.

  Suppose next that $t=s^\omega$. Let $v_0=v$ and let $w\in\omo AA$ be
  the $\omega$-word represented by~$s$. Inductively, we apply
  Corollary~\ref{c:lifting-omega-factorizations} to the inequality
  $u\le v_i$ ($i\ge0$) and the factorization $u=wu$, to obtain a
  factorization $v_i=x_iv_{i+1}$ in~\omo AA such that the inequalities
  $w\le x_i$ and $u\le v_{i+1}$ hold in~\pv W. By
  Proposition~\ref{p:repetitions}, there exist $i,j$ such that
  $0\le i \leq j$ and
  \begin{equation}
    v=x_0\cdots x_{i-1}(x_i\cdots x_j)^\omega v_{j+1}.
    \label{eq:left-factorization}
  \end{equation}
  Similarly, by the left-right dual of Proposition~\ref{p:repetitions},
  there exists a factorization $v_{j+1}=v'(y_n\cdots y_m)^\omega
  y_{m-1}\cdots y_0$ in~\omo AA such that each of the inequalities
  $u\le v'$ and $w\le y_p$ ($p=0,\ldots,n$) holds in~\pv W. Combining
  with~\eqref{eq:left-factorization}, we obtain the factorization
  \begin{equation}
    v=x_0\cdots x_{i-1}(x_i\cdots x_j)^\omega v'
    (y_n\cdots y_m)^\omega y_{m-1}\cdots y_0. 
    \label{eq:global-factorization}
  \end{equation}
  Since the pseudovariety of monoids generated by~\pv V is contained
  in~\pv W, it must satisfy the equality $u=v'$, whence \pv V
  satisfies the inequality $v'\le u$. By the choice of~$\Gamma$, it
  follows that the inequality $u^\omega\le u^\omega v'u^\omega$
  belongs to~$\Gamma$. The preceding inequality may be written as
  \begin{equation}
    u\le (w^{j-i+1})^\omega v'(w^{m-n+1})^\omega.
    \label{eq:hypothesis}
  \end{equation}
  Since the inequality~\eqref{eq:hypothesis} belongs to~$\Gamma$ and
  $u=wu=uw$ by aperiodicity, we deduce that
  \begin{equation}
    \Gamma\vdash u\le w^i(w^{j-i+1})^\omega v'(w^{m-n+1})^\omega w^m.
    \label{eq:basic-ineq-consequence}
  \end{equation}
  On the other hand, the term induction hypothesis yields the
  relations
  \begin{equation}
    \Gamma\vdash w\le x_p \text{ and } \Gamma\vdash w\le y_q\quad
    (p=0,\ldots,j; q=0,\ldots,n).
    \label{eq:induction-inequalities}
  \end{equation}
  Combining \eqref{eq:global-factorization},
  \eqref{eq:basic-ineq-consequence}, and
  \eqref{eq:induction-inequalities}, we conclude that $\Gamma\vdash
  u\le v$, which completes the induction step and the proof.
\end{proof}

\section{Main results}
\label{sec:main}

By an \emph{ordered $\omega$-monoid} we mean an ordered monoid with a
unary operation of $\omega$-power such that $s\le t$ implies
$s^\omega\le t^\omega$. A \emph{variety} of ordered $\omega$-monoids
is a class of such structures that is closed under taking homomorphic
images, ordered $\omega$-submonoids, and arbitrary direct products. It
is well known that varieties of ordered $\omega$-monoids are defined
by $\omega$-inequalities
\cite{Burris&Sankappanavar:1981,Keisler:1977}.

For a pseudovariety \pv V of ordered monoids, we denote by $\pv
V^\omega$ the variety of ordered $\omega$-monoids generated by~\pv V.
Taking into account the basis of $\omega$-identities of $\pv A^\omega$
obtained by McCammond~\cite{McCammond:1999a},
Proposition~\ref{p:provability} yields the following basis result.

\begin{Thm}
  \label{t:omega-basis-PolBV}
  Let \pv W be a pseudovariety of aperiodic monoids and let $\pv
  V=\Pol W$. Then the variety $(\PolB V)^\omega$ of ordered
  $\omega$-monoids is defined by the following inequalities:
  \begin{enumerate}
  \item\label{item:omega-basis-PolBV-1} $x(yz)=(xy)z$,
    $x(yx)^\omega=(xy)^\omega x$;
  \item\label{item:omega-basis-PolBV-2}
    $(x^\omega)^\omega=(x^r)^\omega=xx^\omega=x^\omega x = x^\omega$
    for every $r\ge2$;
  \item\label{item:omega-basis-PolBV-3} $u^\omega\le u^\omega
    vu^\omega$ whenever the inequality $v\le u$ is valid in~\pv V.\qed
  \end{enumerate}
\end{Thm}

\begin{proof}
  The equalities in~\eqref{item:omega-basis-PolBV-1}
  and~\eqref{item:omega-basis-PolBV-2} hold in the variety~$\pv
  A^\omega$, whence also in the subvariety~$(\PolB V)^\omega$. The
  inequalities in~\eqref{item:omega-basis-PolBV-3} hold in~$(\PolB
  V)^\omega$ by Proposition~\ref{p:basis-PolB}. That every
  inequality valid in~$(\PolB V)^\omega$ is a consequence of the
  inequalities
  \eqref{item:omega-basis-PolBV-1}--\eqref{item:omega-basis-PolBV-3}
  follows from Proposition~\ref{p:provability}.
\end{proof}

The following is the announced recursive enumerability of the
$\omega$-identity problem for suitable pseudovarieties of ordered
monoids. It is formulated as a transfer result of that property along
the Boolean-polynomial closure.

\begin{Thm}
  \label{t:main-tool}
  Let \pv W be a pseudovariety of aperiodic monoids and let $\pv
  V=\Pol W$. If the $\omega$-inequality problem for \pv V is
  recursively enumerable, then so is it for~$\PolB V$.
\end{Thm}

\begin{proof}
  Consider the set $\Gamma$ defined in the statement of
  Proposition~\ref{p:provability}. By Proposition~\ref{p:provability},
  each $\omega$-inequality over~\pv A valid in $\pv W=\PolB V$ is
  provable from~$\Gamma$. Since the converse follows from
  Proposition~\ref{p:basis-PolB}, we conclude that the
  $\omega$-inequalities over~\pv A that are valid in~\pv W are
  precisely those that are provable from~$\Gamma$. As $\Gamma$ is
  recursively enumerable by hypothesis, the inequalities that are
  provable from~$\Gamma$ constitute a recursively enumerable set.
  Hence the $\omega$-inequality problem for~\pv W over~\pv A is
  recursively enumerable, and therefore the $\omega$-inequality
  problem for~\pv W is also recursively enumerable by
  Proposition~\ref{p:raising-ineq-prob}.
\end{proof}

Note that, from the definitions, it follows immediately that, if \pv V
is a recursively enumerable pseudovariety of ordered monoids, then so
is \PolB V.
Combining Theorem~\ref{t:main-tool} with
Proposition~\ref{p:re-negative-cases}, we obtain the following main
theorem of this paper.

\begin{Thm}
  \label{t:main}
  Let \pv W be a recursively enumerable pseudovariety of aperiodic
  monoids and let $\pv V=\Pol W$. If the $\omega$-inequality
  problem for \pv V is decidable, then so is it for~$\PolB V$.\qed
\end{Thm}

In particular, this leads to the following application.

\begin{Cor}
  \label{c:main1}
  Let \pv V be a recursively enumerable pseudovariety of aperiodic
  monoids such that the $\omega$-inequality problem is decidable for
  \Pol V. Then, the $\omega$-inequality problem is decidable for every
  level and half level of the concatenation hierarchy starting
  from~\pv V.\qed
\end{Cor}

To obtain our main application, we need to be able to start the
induction process with a pseudovariety of the form \Pol W for which
the $\omega$-inequality problem is decidable. For the
Straubing-Th\'erien hierarchy, we should start with \pv W as the
trivial pseudovariety, for which \Pol W is the pseudovariety of
ordered monoids defined by the inequality $1\le x$
\cite[Proposition~8.4]{Pin&Weil:1994c}, and sometimes and henceforth
denoted $\pv J^+$. The join with its dual, defined by the inequality
$1\ge x$, is the pseudovariety \pv J of all finite \Cl J-trivial
monoids. The $\omega$-equality problem for~\pv J has been solved by
the first author~\cite{Almeida:1990b}. The solution consists in
writing $\omega$-words in a canonical form, namely products of letters
and $\omega$-powers of words that are products of distinct letters in
increasing order (assuming a total order on the alphabet), in such a
way that any factor adjacent to an $\omega$-power $u^\omega$ has at
least one letter that does not appear in~$u$.

Given a word $w=a_1\cdots a_n\in A^*$, with $a_i\in A$, we denote by
$w{\uparrow}$ the language $A^*a_1A^*\cdots a_nA^*$, consisting of all
words that admit $w$ as a subword. We also say that $w$ is
a \emph{subword} of a pseudoword $u\in\Om AM$ if
$u\in\overline{w{\uparrow}}$. If $u\in\omo AM$ is an $\omega$-word
in canonical form over~\pv J then $w$ is a subword of~$u$
if and only if it is a subword of $u^{(k)}$ for some positive integer
$k$, where $u^{(k)}$ is obtained from $u$ by replacing each exponent
$\omega$ by~$k$ \cite[Lemma~8.2.3]{Almeida:1994a}, or equivalently,
if $w$ is a subword of some word from the language $u^{(*)}$ defined
by the regular expression obtained from $u$ by replacing each exponent
$\omega$ by~$*$.

\begin{Prop}
  \label{p:V1/2}
  The $\omega$-inequality problem for~$\pv J^+$ is decidable.
\end{Prop}

\begin{proof}
  Let $u\le v$ be an $\omega$-inequality. Since the transformation of
  an $\omega$-word to its canonical form over \pv J is effective,
  without loss of
  generality we may assume that $u$ and $v$ are in canonical form.

  Since every $\pv J^+$-recognizable language is a finite union of
  languages of the form $w{\uparrow}$, we conclude from
  Proposition~\ref{p:order-vs-languages} that $\pv J^+$ satisfies the
  inequality $u\le v$ if and only if every subword of $u$ is also a
  subword of~$v$. In view of the observation preceding the statement
  of this proposition, the latter condition can be decided by checking
  whether the regular language consisting of all subwords of $u^{(*)}$
  is a subset of the regular language of subwords of $v^{(*)}$.
\end{proof}

Finally, here is the main application of our results.

\begin{Cor}
  \label{c:main2}
  The $\omega$-inequality problem is decidable for all members of
  the Straubing-Th\'erien hierarchy.\qed
\end{Cor}

\subsection*{Acknowledgments}

The first author was partially supported by CMUP (UID/MAT/00144/2013),
which is funded by FCT (Portugal) with national (MCTES) and European
structural funds (FEDER), under the partnership agreement PT2020.

The second and third authors were partially supported by the Grant
15-02862S of the Grant Agency of the Czech Republic.

\bibliographystyle{amsplain}

\begin{thebibliography}{10}

\bibitem{Almeida:1990b}
J.~Almeida, \emph{Implicit operations on finite {${\cal {J}}$}-trivial
  semigroups and a conjecture of {I}. {S}imon}, J. Pure Appl. Algebra
  \textbf{69} (1990), 205--218.

\bibitem{Almeida:1994a}
\bysame, \emph{Finite semigroups and universal algebra}, World Scientific,
  Singapore, 1995, {E}nglish translation.

\bibitem{Almeida:1999b}
\bysame, \emph{Hyperdecidable pseudovarieties and the calculation of semidirect
  products}, Int. J. Algebra Comput. \textbf{9} (1999), 241--261.

\bibitem{Almeida:2003cshort}
\bysame, \emph{Profinite semigroups and applications}, Structural theory of
  automata, semigroups and universal algebra (New York) (V.~B. Kudryavtsev and
  I.~G. Rosenberg, eds.), Springer, 2005, pp.~1--45.

\bibitem{Almeida&Bartonova&Klima&Kunc:2015}
J.~Almeida, J.~Barto{\v{n}}ov{\'a}, O.~Kl{\'i}ma, and M.~Kunc, \emph{On
  decidability of intermediate levels of concatenation hierarchies},
  Developments in Language Theory (I.~Potapov, ed.), Lect. Notes in Comput.
  Sci., no. 9168, 2015, pp.~58--70.

\bibitem{Almeida&Cano&Klima&Pin:2015}
J.~Almeida, A.~Cano, O.~Kl{\'\i}ma, and J.-{\'E}. Pin, \emph{On fixed points of
  the lower set operator}, Int. J. Algebra Comput. \textbf{25} (2015),
  259--292.

\bibitem{Almeida&ACosta:2007a}
J.~Almeida and A.~Costa, \emph{Infinite-vertex free profinite semigroupoids and
  symbolic dynamics}, J. Pure Appl. Algebra \textbf{213} (2009), 605--631.

\bibitem{Almeida&Costa&Zeitoun:2009a}
J.~Almeida, J.~C. Costa, and M.~Zeitoun, \emph{Iterated periodicity over finite
  aperiodic semigroups}, European J. Combin. \textbf{37} (2014), 115--149.

\bibitem{Almeida&Costa&Zeitoun:2015}
\bysame, \emph{{M}c{C}ammond's normal forms for free aperiodic semigroups
  revisited}, LMS J. Comput. Math. \textbf{18} (2015), 130--147.

\bibitem{Almeida&Costa&Zeitoun:2015a}
\bysame, \emph{Factoriality and the {P}in-{R}eutenauer procedure}, Discrete
  Math. \& Theor. Comp. Sci. \textbf{18} (2016).

\bibitem{Almeida&Steinberg:2000a}
J.~Almeida and B.~Steinberg, \emph{On the decidability of iterated semidirect
  products and applications to complexity}, Proc. London Math. Soc. \textbf{80}
  (2000), 50--74.

\bibitem{Burris&Sankappanavar:1981}
S.~Burris and H.~P. Sankappanavar, \emph{A course in universal algebra}, Grad.
  Texts in Math., no.~78, Springer, Berlin, 1981.

\bibitem{Cohen&Brzozowski:1971}
R.~S. Cohen and J.~A. Brzozowski, \emph{Dot-depth of star-free events}, J.
  Comput. System Sci. \textbf{5} (1971), 1--15.

\bibitem{Eilenberg:1976}
S.~Eilenberg, \emph{Automata, languages and machines}, vol.~B, Academic Press,
  New York, 1976.

\bibitem{Gool&Steinberg:2016}
S.~J.~van Gool and B.~Steinberg, \emph{Pro-aperiodic monoids via saturated
  models}, Tech. report, 2016, arXiv:1609.07736.

\bibitem{Keisler:1977}
H.~J. Keisler, \emph{Fundamentals of model theory}, Handbook of Mathematical
  Logic (J.~Barwise, ed.), Studies in Logic and the Foundations of Mathematics,
  vol.~90, North Holland, Amsterdam, 1977, pp.~47--104.

\bibitem{McCammond:1999a}
J.~McCammond, \emph{Normal forms for free aperiodic semigroups}, Int. J.
  Algebra Comput. \textbf{11} (2001), 581--625.

\bibitem{McKnight&Storey:1969}
J.~D. McKnight, Jr. and A.~J. Storey, \emph{Equidivisible semigroups}, J.
  Algebra \textbf{12} (1969), 24--48.

\bibitem{Molchanov:1994}
V.~A. Molchanov, \emph{Nonstandard characterization of pseudovarieties},
  Algebra Universalis \textbf{33} (1995), 533--547.

\bibitem{Pin:1995d}
J.-E. Pin, \emph{Eilenberg's theorem for positive varieties of languages},
  Russian Math. (Iz. VUZ) \textbf{39} (1995), 74--83.

\bibitem{Pin:1997}
\bysame, \emph{Syntactic semigroups}, Handbook of Formal Languages
  (G.~Rozenberg and A.~Salomaa, eds.), Springer, 1997.

\bibitem{Pin&Weil:1996a}
J.-E. Pin and P.~Weil, \emph{Profinite semigroups, {M}al'cev products and
  identities}, J. Algebra \textbf{182} (1996), 604--626.

\bibitem{Pin&Weil:1996b}
\bysame, \emph{A {R}eiterman theorem for pseudovarieties of finite first-order
  structures}, Algebra Universalis \textbf{35} (1996), 577--595.

\bibitem{Pin&Weil:1994c}
\bysame, \emph{Polynomial closure and unambiguous product}, Theory Comput.
  Syst. \textbf{30} (1997), 383--422.

\bibitem{Place&Zeitoun:2014b}
T.~Place and M.~Zeitoun, \emph{Going higher in the first-order quantifier
  alternation hierarchy on words}, Automata, languages, and programming. {P}art
  {II} ({ICALP'14}) (J.~Esparza, P.~Fraigniaud, T.~Husfeldt, and
  E.~Koutsoupias, eds.), Lect. Notes in Comput. Sci., vol. 8573, 2014,
  pp.~342--353.

\bibitem{Place&Zeitoun:2014a}
\bysame, \emph{Separating regular languages with first-order logic},
  {CSL-LICS'14}, 2014, DOI 10.1145/2603088.2603098.

\bibitem{Reiterman:1982}
J.~Reiterman, \emph{The {B}irkhoff theorem for finite algebras}, Algebra
  Universalis \textbf{14} (1982), 1--10.

\bibitem{Rhodes&Steinberg:2009qt}
J.~Rhodes and B.~Steinberg, \emph{The $q$-theory of finite semigroups},
  Springer Monographs in Mathematics, Springer, 2009.

\bibitem{Straubing:1981a}
H.~Straubing, \emph{A generalization of the {S}ch\"utzenberger product of finite
  monoids}, Theor. Comp. Sci. \textbf{13} (1981), 137--150.

\bibitem{Straubing:1985}
\bysame, \emph{Finite semigroup varieties of the form ${V}*{D}$}, J. Pure Appl.
  Algebra \textbf{36} (1985), 53--94.

\bibitem{Thomas:1982}
W.~Thomas, \emph{Classifying regular events in symbolic logic}, J. Comput.
  System Sci. \textbf{25} (1982), 360--376.

\bibitem{Therien:1981a}
D.~Th\'erien, \emph{Classification of finite monoids: the language approach},
  Theor. Comp. Sci. \textbf{14} (1981), 195--208.

\bibitem{Willard:1970}
S.~Willard, \emph{General topology}, Addison-Wesley, Reading, Mass., 1970.

\end{thebibliography}

\providecommand{\bysame}{\leavevmode\hbox to3em{\hrulefill}\thinspace}
\providecommand{\MR}{\relax\ifhmode\unskip\space\fi MR }
% \MRhref is called by the amsart/book/proc definition of \MR.
\providecommand{\MRhref}[2]{%
  \href{http://www.ams.org/mathscinet-getitem?mr=#1}{#2}
}
\providecommand{\href}[2]{#2}

\end{document}